\documentclass[12pt, a4paper]{article}
\usepackage[english]{babel}
\usepackage{amsmath, amsthm, amsfonts, amsbsy, amssymb, scrextend, graphicx}
\usepackage{hyperref} 
\usepackage{verbatim}
\usepackage{color}
\usepackage{systeme}
\usepackage{multirow}
\usepackage{tikz}
\usepackage[export]{adjustbox}
\usepackage{float}

 \newcommand{\R}{\mathbb{R}}
 \newcommand{\Z}{\mathbb{Z}}
 
\newcommand{\E}{\mathsf{E}}
\newcommand{\Var}{\mathsf{Var}}

 \newcommand{\N}{\mathbb{N}}
\renewcommand{\P}{\mathsf{P}}

 \newcommand{\bx}{{\bf x}}

 \newcommand{\bz}{{\bf z}}

\newcommand{\eps}{\varepsilon}

 \newtheorem{theorem}{Theorem}[section]
 \newtheorem{lemma}{Lemma}[section]
 
 \newtheorem{remark}{Remark}
\newtheorem{proposition}{Proposition}[section]
\newtheorem{definition}{Definition}[section]

\numberwithin{equation}{section}

\begin{document}


\title{Balls-in-bins models with asymmetric feedback and reflection}

\author{
Mikhail Menshikov
\footnote{Department of Mathematical Sciences, Durham University, UK. \newline
\indent  Email: mikhail.menshikov@durham.ac.uk
}\\
{\small  Durham University}
\and 
Vadim Shcherbakov
\footnote{Department of Mathematics, Royal Holloway,  University of London, UK. \newline
\indent  Email: vadim.shcherbakov@rhul.ac.uk
}\\
{\small  Royal Holloway,  University of London}
}

\maketitle



\begin{abstract}
{\small 
Balls-in-bins  models  describe a random sequential allocation of infinitely many  balls 
into a finite number of  bins. In these models a  ball is placed into  a bin with probability 
proportional to a given function (feedback function), which depends
 on the number of existing balls in the bin. Typically, the feedback function 
 is the same for all bins (symmetric feedback), and  there are no 
constraints on the number of balls in the bins.
In this paper  we study versions of BB models with two bins, in which 
the above assumptions are violated.  
In the first model  of interest the feedback functions can depend on a bin
(BB model with asymmetric feedback).
In the case when both feedback functions are power law and  superlinear, 
 a single bin receives all but finitely many balls almost surely, and we study the probability that 
 this happens for a given bin. In particular, 
under certain initial conditions we derive the normal approximation for  this  probability.
This generalizes  the  result in~\cite{Spencer} obtained  in the case of the symmetric feedback.
The main part of the paper concerns
 the BB model with asymmetric feedback evolving 
subject to certain constraints on the  numbers of allocated balls. 
The  model can  be  interpreted as a transient reflecting  random walk 
in a curvilinear wedge, and we obtain  a complete  classification  of its 
 long term behavior. 
}
\end{abstract}


\section{Introduction}

A balls-in-bins (BB) model  is a classic probabilistic 
model describing a random  sequential allocation of infinitely many balls into a finite number of bins.
The probability of allocating a ball into a bin is proportional to a function, which depends on
 the number of existing balls in the bin. 
In the standard setup, the feedback function is the same for all bins (i.e. 
the feedback is symmetric), and there are no constraints on the number of allocated balls.

This paper concerns  two models of random sequential allocation of balls into bins,
which  can be regarded as versions of commonly studied  BB models 
with a power  law feedback function $f(x)=x^{\beta}$, where $\beta>0$. 
The following results for such s BB model are known (e.g. see~\cite{Spencer},~\cite{Oliveira} and references therein). If $\beta\leq 1$, then, with probability one,
all bins receive infinitely many balls. For example, if $\beta=1$ (the  classical P\'olya  urn model), then 
the  distribution of  fractions of balls in the  bins converges to 
a  Dirichlet  distribution.
On the other hand, if $\beta>1$, then, with probability one,   a single random bin  
receives  all but finitely many balls. 

Our first model of  interest  goes as follows.
Suppose that there are  two bins labeled  1 and 2 respectively.
 If $x_i$ is the number of existing balls in  the bin $i=1,2$,
 then the next ball 
is placed in the bin $i$ with probability proportional to $x_i^{\beta_i}$, where 
$\beta_i$, $i=1,2$ are given positive constants.
If $\beta_1=\beta_2$ (the symmetric case), then this is the model studied
in~\cite{Spencer}.
Here we focus on the asymmetric case $\beta_1\neq \beta_2$ and both 
$\beta_1>1$ and $\beta_2>1$, in which, similar to the symmetric one,  all but finitely many 
 balls are placed in a single random bin almost surely, and we study the probability 
 that this  will happen in a given bin.
The probability  depends on initial conditions, and is trivially equal to $1$ or $0$ (depending on a bin) in most cases. There is, however, a subset of initial conditions, for which this probability is non-trivial,
and we derive the normal  approximation for it. 
This generalizes  the  result in~\cite{Spencer},  where the normal approximation 
for this probability is obtained  in the symmetric case.

The main results of this paper concern the other model.
This model can be thought of as the BB model with asymmetric feedback, where 
the balls are allocated subject to certain  constraints.
The model can be naturally  interpreted as a  reflecting  random walk 
in a curvilinear wedge, and   in what follows we call it the BB model with  reflection.

We obtain a complete  classification  of the long term behavior of the BB model with reflection 
 in the case when the boundaries of the wedge are given by power law functions.
Briefly this can be described as follows.
Let $\zeta_i(t)$ be the number of balls in the bin $i=1,2$ at time $t$.
Then, with probability one, the process $(\zeta_1(t), \zeta_2(t))$ 
 eventually confines 
to a strip of a finite width along the boundary of the wedge.  Moreover, we  exactly
calculate the width of the strip. The width is an integer that 
 depends on the model parameters. 
This effect is somewhat  similar to boundary effects  detected in 
competition processes  in~\cite{MS18} and~\cite{SV19}.

Our motivation for the models in this paper has arisen from previous studies 
of growth models that  can be thought of reinforced urn models with graph based interaction.
In these models bins are labeled by vertices of a graph, and allocation probabilities at a bin 
 depend on the configuration of allocated balls in the neighbourhood of that bin.
 Several such growth models have been studied 
in the case where the allocation probability at a bin  is proportional  
to a log-linear  function of the  local configuration  (see~\cite{CMSV}, \cite{MS20}
and\cite{SV10}). A strong reinforcement mechanism in those models 
allowed to characterize precisely their  limit behavior. 
In contrast to the growth models with the log-linear interaction, 
the growth models with interaction given by power law  functions are harder 
to analyze. The  BB model with the power law asymmetric feedback can be regarded as 
 such a growth model {\it without} interaction. 
Secondly, the BB model with reflection 
can be regarded as a toy model mimicking the behavior of one of these growth 
models in  a very particular situation.

The rest of the paper is organized as follows. In Section~\ref{results}
we formally define both  BB models of interest   and state the main results.
 In Section~\ref{prelim1} we collect
some known  facts about pure birth processes with continuous time. 
These facts are  used in Section~\ref{proofs-no-reflection}
 to study the  long term behavior of the BB model with asymmetric feedback.
In Sections~\ref{lower}-\ref{upper} we prove a series of lemmas  that describe 
the most likely long term behavior of the BB model with reflection. 
These  statements are used in Section~\ref{proof-reflection-theorems}
to prove the main  results (Theorems~\ref{main-sym} and~\ref{main-asym} below) 
for the BB model with reflection.
Finally,  in Section~\ref{open} we state an open problem.

\section{Models and results}
\label{results}

Start with some notations and elementary facts.
 Let $\N=\{1,2....\}$ and $\Z_{+}=\{0,1,2,...\}$ be  sets of natural numbers and non-negative numbers 
 respectively. We assume that all random variables and processes are realised 
on a probability space $(\Omega, {\cal F}, \P)$. The expectation with respect to the probability 
$\P$ is denoted by $\E$. 
Given a Markov process with values  in $\N^2$  we denote by  $\P_{\bz}$ the 
distribution of the process  started at $\bz\in \N^2$.
Given quantities $a(m)$ and $b(m)$ depending on $m\in \N$, we write 
  $a(m)\sim b(m)$ to denote  the fact that $a(m)/b(m)\to 1$, as $m\to \infty$.

Throughout we deal   with many analytical equations involving integer valued  quantities.
This formally requires using the floor function and   the ceiling function.  
However,  most of the  equations are of interest for us  only asymptotically, 
when  the effect of applying  floor/ceiling functions can be neglected.
Therefore, to simplify notations, we  skip writing these functions.

\subsection{The model with asymmetric feedback}

In this section we define  the  BB model with asymmetric feedback and state the corresponding results.
The model is as follows. 
Infinitely many balls are sequentially placed 
in two available bins according to the following rule. 
If $x_i$ is the number of existing  balls in the bin $i=1,2$, then the next arriving ball 
is placed in this bin with probability proportional to $x_i^{\beta_i}$, where $\beta_1>0$ and $\beta_2>0$ are given constants. 
Let $X_i(t)$ be  the number of balls in the bin $i=1,2$ at time $t\in\Z_{+}$.
Then the random process $X(t)=(X_1(t), X_2(t))\in \N^2$, $t\in \Z_{+}$, is  a discrete time Markov chain (DTMC) 
 with the  transition probabilities
\begin{equation}
\label{prob1}
\begin{split}
\P_{\bx}(X_1(1)=x+1, X_2(1)=y)&=
\frac{x^{\beta_1}}{x^{\beta_1}+y^{\beta_2}},\\
\P_{\bx}(X_1(1)=x, X_2(1)=y+1)&=
\frac{y^{\beta_2}}{x^{\beta_1}+y^{\beta_2}},
\end{split}
\end{equation}
for $\bx=(x,y)$.

Define events
\begin{equation}
\label{A1}
A_i=\left\{\text{\it all but finitely many balls are placed in  the bin } i\right\},\quad i=1,2.
\end{equation}

\begin{theorem}
\label{T0}

1) If both $\beta_1>1$ and $\beta_2>1$, then 
$\P_{\bx}\left(A_1\right)+\P_{\bx}\left(A_2\right)=1$
for any  $\bx\in\N^2$.

2) If $\beta_1\leq 1$ and $\beta_2>1$,
then $\P_{\bx}\left(A_2\right)=1$ for any  $\bx\in\N^2$.
\end{theorem}

\begin{remark}
{\rm 
In other words,  Theorem~\ref{T0} states that, if at least one of  $\beta$-parameters  is greater than one, then
the monopolistic regime realizes almost surely. The theorem follows from the well known argument, that consists in
embedding  the DTMC $X(t)$ into a pair of continuous time pure birth processes
(see Section~\ref{prelim1} for details).
}
\end{remark}

Given  $1<\beta_1\leq \beta_2$, define
\begin{equation}
\label{alpha-gamma}
\alpha_{cr}=\frac{\beta_1-1}{\beta_2-1}
\quad\text{and}\quad
\nu_{cr}=
\alpha_{cr}^{\frac{1}{\beta_2-1}}
=\left(\frac{\beta_1-1}{\beta_2-1}\right)^{\frac{1}{\beta_2-1}}.
\end{equation}

 \begin{theorem}
\label{T1a}
Let $X(t)=(X_1(t), X_2(t))$ be the DTMC with transition probabilities~\eqref{prob1},
where $1<\beta_1\leq \beta_2$. 
Assume that  $\bx=(x, y)\in\N^2$, where 
$$y=y(x)=\nu x^{\alpha}+o(x^{\alpha}),\quad\text{as}\quad x\to \infty,$$
for some  $\alpha\in (0,1)$ and $\nu>0$.
\begin{enumerate}
\item[1)] If either $\alpha<\alpha_{cr}$, or  
$\alpha=\alpha_{cr}$ and $\nu<\nu_{cr}$, 
then  
$
\lim_{x\to\infty}\P_{\bx}\left(A_1\right)=1$.
\item[2)] If  either $\alpha>\alpha_{cr}$, or $\alpha=\alpha_{cr}$ and $\nu>\nu_{cr}$, then 
$\lim_{x\to\infty}\P_{\bx}\left(A_2\right)=1$.
\end{enumerate}
\end{theorem}

\begin{remark}
{\rm 
Note that the curve 
\begin{equation}
\label{curve}
y=\nu_{cr}x^{\alpha_{cr}}=\alpha_{cr}^{\frac{1}{\beta_2-1}}x^{\frac{\beta_1-1}{\beta_2-1}}
\end{equation}
determines a phase transition in the long term behavior of the process.
The curve  has a natural  probabilistic interpretation explained in Remark~\ref{cr-curve-sense} below.
}
\end{remark}

Theorem~\ref{T1b} below deals with the  critical case $\alpha=\alpha_{cr}, \nu=\nu_{cr}$ 
not covered by Theorem~\ref{T1a}.

\begin{theorem}
\label{T1b}
Let $X(t)=(X_1(t), X_2(t))$ be the DTMC with transition probabilities~\eqref{prob1}, where 
 $1<\beta_1\leq \beta_2$.
Assume that  $\bx=(x, y)\in\N^2$ is such that 
$y=y(x)=\nu_{cr}x^{\alpha_{cr}}+\mu x^{\delta}+o(x^{\delta})$,
as $x\to \infty,$
for some  $\mu\in\R$ and $\delta\in [0, \alpha_{cr})$, and let 
\begin{equation}
\label{rho}
\rho=
-\mu\left(\frac{2\beta_2-1}{\nu_{cr}}\right)^{\frac{1}{2}}.
\end{equation}
\begin{enumerate}
\item[ 1)] If $\delta=\frac{\alpha_{cr}}{2}$, then 
$$\lim_{x\to\infty}\P_{\bx}\left(A_1\right)=\Phi(\rho)\quad\text{and}\quad 
\lim_{x\to\infty}\P_{\bx}\left(A_2\right)=1-\Phi(\rho),$$ 
where $\Phi$ is the cumulative distribution function of the standard normal distribution.
\item[2)] If  $\delta\in\left[0,\frac{\alpha_{cr}}{2}\right)$, then 
$\lim_{x\to\infty}\P_{\bx}\left(A_1\right)=
\lim_{x\to\infty}\P_{\bx}\left(A_2\right)
=\frac{1}{2}.
$
\item[3)] If  $\delta\in\left(\frac{\alpha_{cr}}{2}, \alpha_{cr}\right)$ and $\mu\neq 0$, then
\begin{enumerate}
\item $\lim_{x\to\infty}\P_{\bx}\left(A_1\right)=0$ for $\mu>0$, 
\item  $\lim_{x\to\infty}\P_{\bx}\left(A_1\right)=1$ for $\mu<0$.
\end{enumerate}
\end{enumerate}
\end{theorem}

\begin{remark}
\label{Spencer-R}
{\rm
If $\beta_1=\beta_2=\beta>1$, and 
$X_1(0)=a$  and $X_2(0)=a-\frac{\rho}{\sqrt{2\beta-1}}\sqrt{a}$, then, 
by  Part 1) of Theorem~\ref{T1b}, the  
 probability of event $A_1$  is approximated by $\Phi(\rho) $
  for sufficiently large $a$.
This recovers  the result of \cite{Spencer}, where  they obtained the  normal approximation $\Phi(\rho)$
for the same probability in the symmetric case 
 assuming that the process starts at  $(x, y)$, where 
$x=a+\frac{\rho}{\sqrt{4\beta-2}}\sqrt{a}$ and
$y=a-\frac{\rho}{\sqrt{4\beta-2}}\sqrt{a}$.
The minor difference in  initial conditions is just due to our attempt
to somehow accommodate  both the symmetric and the asymmetric case in a single statement
(clearly, other variants of the initial condition are possible). 
 }
\end{remark}

\subsection{The model with reflection}

In this section we define  the  BB model with reflection  and state the corresponding results.

Given two monotonically increasing and  sufficiently smooth 
 functions $\varphi, \psi: \R_{+}\to \R_{+}$, such that $\varphi(x)<\psi(x)$ for all $x\geq 0$,
  define the domain 
\begin{equation}
\label{domain}
Q_{\varphi, \psi}:=\{(x, y)\in\N^2: \varphi(x)\leq y\leq \psi(x)\}.
\end{equation}
Given constants  $\beta_1>0$ and $\beta_2>0$ consider a DTMC 
$\zeta(t)=(\zeta_1(t), \zeta_2(t))\in Q_{\varphi, \psi},\, t\in \Z_{+}$
with the following transition probabilities  
\begin{equation}
\label{prob21-22}
\begin{split}
\P_{\zeta}(\zeta_1(1)=&\, x+1,\zeta_2(1)=y)=\frac{x^{\beta_1}}{x^{\beta_1}+y^{\beta_2}}\\
\P_{\zeta}(\zeta_1(1)=&\, x, \zeta_2(1)=y+1)=\frac{y^{\beta_2}}{x^{\beta_1}+y^{\beta_2}}\\
&\text{for}\quad \zeta=(x,y): \varphi(x+1)\leq y\leq \psi(x)-1;
\end{split}
\end{equation}
and 
\begin{equation}
\label{prob23}
\begin{split}
\P_{\zeta}(\zeta_1(1)=&\, x, \zeta_2(1)=y+1)=1\\
&\text{for}\quad \zeta=(x,y):\varphi(x)\leq y< \varphi(x+1),
\end{split}
\end{equation}
\begin{equation}
\label{prob24}
\begin{split}
\P_{\zeta}(\zeta_1(1)=&\, x+1, \zeta_2(1)=y)=1\\
&\text{for}\quad \zeta=(x,y): \psi(x)-1<y\leq\psi(x).
\end{split}
\end{equation}

\begin{definition}
\label{jumps-reflections}
{\rm 
\begin{itemize}
\item 
We call the  Markov chain $\zeta(t)$ with transition probabilities given in equations~\eqref{prob21-22}-\eqref{prob24}
 the BB model with reflection in the domain $Q_{\varphi, \psi}$. 
Curves  $y=\varphi(x)$ and $y=\psi(x)$ are called the lower boundary and the upper boundary respectively. 
\item 
The one-step transitions 
of the DTMC $\zeta(t)$   described by equations~\eqref{prob21-22} are called 
  jumps (jumps to the right and jumps up), 
and the one-step transitions described by equations~\eqref{prob23}--\eqref{prob24} are called reflections (by the lower and by the upper boundary respectively).
\item 
We say that the process is reflected by the lower (upper) boundary, when 
the transition described by~\eqref{prob23} (by~\eqref{prob24}) occurs.
\end{itemize}
}
\end{definition}

\begin{remark}
{\rm Note that, by  construction,  in the interior of the domain  $Q_{\varphi,\psi}$
 the DTMC $\zeta(t)$ evolves  
  in the same way 
as  the DTMC   $X(t)$ with transition probabilities~\eqref{prob1}.
Equations~\eqref{prob23}--\eqref{prob24}   
describe a reflection mechanism  preventing the process from  leaving the domain.
}
\end{remark}
\begin{remark}
{\rm 
It should be also noted  that our results for the DTMC $\zeta(t)$ 
concern only the case, when the boundaries are given by power
law functions, that is $\varphi(x)=x^{\alpha}$ and $\psi(x)=x^{\gamma}$ for some 
$0<\alpha<\gamma$. However,  it is convenient 
 to have the model domain  defined in the general case.
First, it is convenient for  future references (e.g. Section~\ref{open}).
Also,  conditions in both~\eqref{prob23} and~\eqref{prob24} 
look   more transparent  (at least, as it seems to us)   in the general case.
}
\end{remark}

Let $\zeta(t)$ be the DTMC with transition probabilities~\eqref{prob21-22}-\eqref{prob24}.
Define events
\begin{equation}
\label{Ak}
\begin{split}
A_{\alpha, k}=&\{\limsup_{t\to \infty}(\zeta_2(t)-\zeta_1^{\alpha}(t))=k\}\cap
\{\liminf_{t\to \infty}(\zeta_2(t)-\zeta_1^{\alpha}(t))=0\}\\
&
\text{ for}\quad \alpha\in (0,1)\quad \text{and}\quad  k\in \N.
\end{split}
\end{equation}

\begin{equation}
\label{Ak-tilde}
\begin{split}
\widetilde{A}_{\gamma, k}=&
\{\limsup\limits_{t\to \infty}
(\zeta_1(t)-\zeta_2^{1/\gamma}(t))=k\}\cap
\{\liminf\limits_{t\to \infty}(\zeta_1(t)-\zeta_2^{1/\gamma}(t))=0\}\\
& \text{ for}\quad\gamma\geq 1\quad\text{and}\quad k\in \N.
\end{split}
\end{equation}

\begin{equation}
\label{Ek}
\begin{split}
E_{k}=&
\{\limsup\limits_{t\to \infty}(\zeta_1(t)-\zeta_2(t))=k\}\cap
\{\zeta_1(t)-\zeta_2(t)=0 \text{ i.o.}\}\cap 
\{\zeta_1(t)-\zeta_2(t)=k\text{ i.o.}\}\\
&\text{ for}\quad k\in \N.
\end{split}
\end{equation}

\begin{equation}
\label{Bgamma}
 B_{\gamma}=\{\limsup_{t\to \infty}(\zeta_1^{\gamma}(t)-\zeta_2(t))=1\}\cap
\{\liminf_{t\to \infty}(\zeta_1^{\gamma}(t)-\zeta_2(t))=0\}
\quad\text{for}\quad\gamma\in (0,1).
\end{equation}

\begin{remark}
{\rm 
All events in~\eqref{Ak}-\eqref{Bgamma}  can be easily interpreted geometrically. 
For example, the event $A_{\alpha, k}$ 
 means that given any arbitrary small $\eps>0$  the process eventually confines to the strip
$\{(\zeta_1, \zeta_2): \zeta_1^{\alpha}\leq \zeta_2\leq \zeta_1^{\alpha}+k+\eps\}$, where 
it  fluctuates by
  approaching both the lower boundary $y=x^{\alpha}$  and the curve $y=x^{\alpha}+k$
 at any arbitrarily small distance. 
Note also the special case in~\eqref{Ek}, where 
the event   $E_{k}$  means that the process eventually confines to the 
strip between  the bisector $y=x$ and the straight 
line $y=x-k$, and  visits both of these  lines  infinitely often.
In particular,  the event $E_1$ simply means, that the process 
eventually follows  a "zigzag" trajectory located between the bisector and the line $y=x-1$
(see Remark~\ref{Cor1-bisector}).
}
\end{remark}

Theorems~\ref{main-sym}  and~\ref{main-asym} below are the main  results concerning  the long term behavior
of the model with reflection.

\begin{theorem}
\label{main-sym}
Let $\beta_1=\beta_2=\beta>1$, and 
let $\zeta(t)=(\zeta_1(t), \zeta_2(t))$ be the  BB model with reflection 
in the domain $Q_{\varphi, \psi}$, where 
$\varphi(x)=x^{\alpha}$ and $\psi(x)=x^{\gamma}$ for some $\alpha,  \gamma: 0<\alpha<1<\gamma$.
Then for any $\zeta(0)=\zeta\in Q_{\varphi, \psi}$
 there 
 exist integers  $k_1=k_1(\alpha, \beta)$ and $k_2=k_2(\gamma, \beta)$, such that 
\begin{align}
\label{beta}
\P_{\zeta}\left(A_{\alpha, k_1}\right)+\P_{\zeta}\left(\widetilde{A}_{\gamma, k_2}\right)&=1,
\end{align}
where events  $A_{\alpha, k_1}$ and  $\widetilde{A}_{\gamma, k_2}$ are  
defined in~\eqref{Ak}-\eqref{Ak-tilde},     $k_1$ and $k_2$ are integer 
solutions of  the following inequalities
\begin{equation}
\label{k1-sym}
\frac{\alpha}{(\beta-1)(1-\alpha)}<k_1\leq 1+\frac{\alpha}{(\beta-1)(1-\alpha)}
\end{equation}
and
\begin{equation}
\label{k2-sym}
\frac{1}{(\beta-1)(\gamma-1)}<k_2\leq 1+\frac{1}{(\beta-1)(\gamma-1)}.
\end{equation}
\end{theorem}

\begin{theorem}
\label{main-asym}
Let  $1<\beta_1<\beta_2$, and 
let $\zeta(t)=(\zeta_1(t), \zeta_2(t))$ be the  BB model with reflection 
in the domain $Q_{\varphi, \psi}$, where 
$\varphi(x)=x^{\alpha}$ and $\psi(x)=x^{\gamma}$ for some 
$\alpha, \gamma: 0<\alpha<\alpha_{cr}<\gamma$.
Then, for any $\zeta(0)=\zeta\in Q_{\varphi, \psi}$
  there  exist integers  $k_1=k_1(\alpha, \beta_1, \beta_2)$ and 
$k_2=k_2(\gamma, \beta_1, \beta_2)$, such that 
\begin{align}
\label{gamma<1}
\P_{\zeta}\left(A_{\alpha, k_1}\right)+\P_{\zeta}\left(B_{\gamma}\right)&=1,\quad\text{if}\quad 
\alpha_{cr}<\gamma<1,\\
\label{gamma=1}
\P_{\zeta}\left(A_{\alpha, k_1}\right)+\P_{\zeta}\left(E_{k_2}\right)&=1,\quad\text{if}\quad 
\gamma=1,\\
\label{gamma>1}
\P_{\zeta}\left(A_{\alpha, k_1}\right)+
\P_{\zeta}\left(\widetilde{A}_{\gamma, k_2}\right)&=1,\quad\text{if}\quad \gamma>1,
\end{align}
where  events $A_{\alpha, k}$, $\widetilde{A}_{\gamma, k}$, $E_k$ and $B_{\gamma}$ 
are  defined in~\eqref{Ak}-\eqref{Bgamma}, 
 $k_1$ and $k_2$ are integer solutions of  the following inequalities
\begin{equation}
\label{k1}
\frac{\alpha}{(\beta_2-1)(\alpha_{cr}-\alpha)}<k_1\leq 1+\frac{\alpha}{(\beta_2-1)(\alpha_{cr}-\alpha)}
\end{equation}
\begin{equation}
\label{k2}
\frac{1}{(\beta_2-1)(\gamma-\alpha_{cr})}<k_2\leq 1+\frac{1}{(\beta_2-1)(\gamma-\alpha_{cr})}.
\end{equation}
 \end{theorem}

\begin{remark}
{\rm 
In other words, both Theorem~\ref{main-sym} and Theorem~\ref{main-asym}
 state that, with probability one,   
the DTMC $\zeta(t)$ eventually confines to a strip of a finite width 
along either the upper, or the lower boundary of the process's domain. 
Given the model parameters, the width  of the limit strip can be calculated.

Moreover, 
the process  fluctuates within this limit strip "touching" both boundaries of the strip infinitely many times. 
 The  main difference between these two theorems is that 
in the symmetric case (Theorem~\ref{main-sym})  the critical curve coincides with the bisector.
In this case events~\eqref{Ek} and~\eqref{Bgamma} are impossible. 
}
\end{remark}

\begin{remark}
{\rm 
Note that the inequality~\eqref{k2} for determining $k_2$ in~\eqref{gamma=1}
can be simplified. Indeed, if $\gamma=1$, then
$\frac{1}{(\beta_2-1)(1-\alpha_{cr})}=\frac{1}{\beta_2-\beta_1}$, and we get the 
inequality 
$\frac{1}{\beta_2-\beta_1}<k\leq 1+\frac{1}{\beta_2-\beta_1}$ for determining $k_2$. 
For example, if $\beta_2-\beta_1>1$, then $k_2=1$; if $\beta_2-\beta_1=1$, then $k_2=2$, and so on.
}
\end{remark}

\section{Pure birth processes with power law rates}
\label{prelim1}

For the reader's convenience we provide in this section 
some  facts about pure birth processes with continuous time.
The material of the section will used 
to obtain Theorems~\ref{T0},~\ref{T1a} and~\ref{T1b}.

We say that a random variable $\xi$ has exponential distribution with parameter $\lambda>0$
and write $\xi \sim Exp(\lambda)$, if 
$\P(\xi>x)=e^{-\lambda x}$ for $x\geq 0$.
Let $Z(t)=(Z(t),\, t\in \R_{+})$ be a continuous time 
 pure birth process on $\Z_+$ with birth rates $\{\lambda_i>0,\, i\geq 0\}$, 
that is, $Z(t)$ is  a continuous time Markov chain (CTMC)  on $\Z_{+}$ evolving as follows.
The process stays at state $i$ for a period of time  given by a random variable 
  $\xi_i\sim Exp(\lambda_i)$
and then jumps to state $i+1$ with probability one.
 Random variables $(\xi_i,\, i\geq 0)$  are assumed to be  independent.
The random variable 
$\displaystyle{T_k=\sum_{i=k}^{\infty}\xi_i}$
is called  the    time to explosion of the birth process started at $Z(0)=k$.
It is known (e.g. \cite{Feller}) that 
$\P(T_0<\infty)=1$  
 if and only if $\displaystyle{\sum_{i=0}^{\infty}\lambda_i^{-1}<\infty}$.
When $\P(T_0<\infty)=1$ (and, hence, $\P(T_k<\infty)=1$  for $k\in\Z_{+}$)
 we say that the process is explosive with probability $1$.
For the explosion time $T_k$  we  have that 
\begin{equation}
\label{conv2}
\E(T_k)=\sum_{i=k}^{\infty}\E(\xi_i)=\sum_{i=k}^{\infty}\lambda_i^{-1}\quad\text{and}\quad
\Var(T_k)=\sum_{i=k}^{\infty}\Var(\xi_i)=\sum_{i=k}^{\infty}\lambda_i^{-2}.
\end{equation}

Consider a pure birth process  birth rates $\lambda_k=k^{\beta},\, k\geq 1$, where $\beta>1$.
Given the above,  such a  birth process is explosive with probability $1$, and 
 the mean and the variance of the  explosion time $T_k$  in this case are as follows
\begin{align}
\label{mean-var}
m_k&:=\E(T_k)\sim \frac{k^{-\beta+1}}{\beta-1}\quad\text{and}\quad 
s_k^2:=\Var(T_k)\sim  \frac{k^{-2\beta+1}}{2\beta-1}.
\end{align}
Therefore, 
\begin{equation}
\label{mean-var}
\frac{s_k^2}{m_{k}^2}=\frac{O(1)}{k}\to 0, \quad\text{as}\quad k\to \infty,
\end{equation}
and, by Chebyshev's inequality, 
\begin{equation}
\label{cheb1}
\P\left(|T_{k}-\E(T_{k})|\geq \delta\cdot\E(T_k)\right)\leq \frac{O(1)}{\delta^2 k}
\quad\text{for all}\quad\delta>0.
\end{equation}
Rewriting $T_k=\sum\limits_{i=k}^{k^2-1}\xi_i+ \eta_k$, where 
$\eta_k=\sum\limits_{i=k^2}^{\infty}\xi_i$,
 gives the  scheme of series 
$\left(\tilde\xi_{k,i},\, i=1,...,k^2\right)$,\, $k\geq 2$, where 
$\tilde\xi_{k,i}=\xi_i$ for $i=k,..., k^2-1$ and $\tilde\xi_{k,k^2}=\eta_k$, that satisfies 
the Lyapunov's conditions of the Central Limit Theorem.
Therefore, the  random variable 
$\displaystyle{\frac{T_k-m_k}{s_k}}$ converges in distribution 
to the standard normal distribution, as $k\to \infty$.
Moreover, since 
$$\E(\xi^3)=6\lambda^{-3}\quad\text{and}\quad
\E\left(|\xi-\E(\xi)|^3\right)\leq \frac{7}{\lambda^3}\quad\text{for}\quad
\xi\sim Exp(\lambda),$$
we obtain the following for  the  pure birth process with polynomial   rates 
\begin{align}
r_{k}&:=\sum_{i=k}^{\infty}\E\left(|\xi_i-\E(\xi_i)|^3\right)\leq 7\sum_{i=k}^{\infty}i^{-3\beta}
\leq  C_3\frac{k^{-3\beta+1}}{3\beta-1}.
\label{r3}
\end{align}
Combining this with~\eqref{mean-var} gives 
$\displaystyle{\frac{r_{k}}{s_{k}^{3}}\leq \frac{O(1)}{\sqrt{k}}}$, as $k\to \infty$.
Therefore, by  Berry-Esseen theorem (e.g. Theorem 2 in~\cite{Feller}, Section 5, Chapter XVI),
we have  that 
\begin{equation}
\label{Berry}
\sup_{t\in\R}\left|\P\left(\frac{T_k-m_{k}}{s_{k}}\leq t\right)-\Phi(t)\right|
\leq \frac{O(1)}{\sqrt{k}},
\end{equation}
where $\Phi$ is the cumulative distribution function of the standard normal distribution.
\begin{remark}
{\rm 
For comparison, consider a pure birth process with exponential transition rates $\lambda_k=e^{\beta k}$, where
 $\beta>0$. Then, a direct computation gives that  $s_k^2=\Var(T_k)\sim C_2 e^{-2\beta k}$ and 
$r_k\leq C_3e^{-3\beta k}$. Therefore,  $\frac{r_k}{s_k^{3}}\sim C$, and, hence,
 the normal approximation does not apply in this case.
}
\end{remark}

\begin{remark}
\label{cr-curve-sense}
{\rm
The critical curve~\eqref{curve} 
 is obtained by equating  the mean explosion times of the corresponding continuous time 
  birth processes with polynomial rates.
 Note that  the critical  curve lies   below  the curve 
$y=x^{\frac{\beta_1}{\beta_2}}$ obtained by equating the 
 probability  of the  jump up to the probability of the jump to the right. 
 }
 \end{remark}

 Given continuous 
time birth processes $(Z_i(t),\, t\in \R_{+})$, $i=1,2$,  we denote by  $ T_{z,i}$ the explosion time
of the process $Z_i(t)$ started at $z_i\in \N$ and 
define   $V_{x, y}=T_{x,1}-T_{y,2}$ for $x, y\in \N$.

\section{Proofs of Theorems~\ref{T0},~\ref{T1a} and~\ref{T1b}}
\label{proofs-no-reflection}

\subsection{Proof of Theorem~\ref{T0}}
\label{proof-T0}

Theorem~\ref{T0} basically  follows  from the well known 
embedding argument   known  as  Rubin's construction (\cite{Davis}). 
We provide the proof just for the sake of completeness.

 The DTMC $X(t)=(X_1(t), X_2(t))$
 with transition probabilities~\eqref{prob1} can be regarded as the embedded Markov chain 
 for  a pair of independent continuous time pure birth 
  processes 
$Z_1(t)$ and $Z_2(t)$ with 
  the polynomial  birth  rates  $k^{\beta_1}$,\, $k\in \Z_{+}$, and
  $k^{\beta_2}$,\, $k\in \Z_{+}$, respectively.
Let $ T_{z,i}$
be  the explosion time of the birth process $Z_i(t)$ started at $z\in \Z_{+}$, $i=1,2$ and recall
events $A_i$ defined in~\eqref{A1}.
Observe  that  
$\P_{\bx}\left(A_1\right)
=\P\left(T_{x,1}<T_{y,2}\right)$ and $\P_{\bx}\left(A_2\right)
=\P\left(T_{x,1}>T_{y,2}\right)$  for $\bx=(x,y)$.
If both $\beta_1>1$ and $\beta_2>1$, then 
   $\P\left(T_{x,1}<\infty\right)=\P\left(T_{y,2}<\infty\right)=1.$
Random variables $T_{x,1}$ and $T_{y,2}$ are independent 
and have absolutely continuous distributions. Therefore, 
 $\P\left(T_{x,1}<T_{y,2}\right)+
\P\left(T_{x,1}>T_{y,2}\right)=1,$
 and, hence, 
 $\displaystyle{\P_{\bx}\left(A_1\right)+\P_{\bx}\left(A_2\right)=1,}$
 as claimed in Part 1) of Theorem~\ref{T0}.
If $\beta_1\leq 1$ and $\beta_2>1$, then 
$\P\left(T_{x,1}=\infty\right)=\P\left(T_{y,2}<\infty\right)=1$, 
and, hence,  $ \P_{\bx}\left(A_2\right)=1$,  as claimed in Part 2) of Theorem~\ref{T0}.

\subsection{Proof of Theorem~\ref{T1a}}
\label{proof-T1a}

The proof of Theorem~\ref{T1a} is rather straightforward and  is based 
on the fact that the explosion times 
can be approximated by their expectations, as  it follows from~\eqref{cheb1}.
We provide main details for the sake of completeness, and prove 
 only Part 1) of the theorem, as Part 2) can be proved similarly.

If   
 $y(x)=\nu x^{\alpha_{cr}}+o(x^{\alpha_{cr}})$, where $\nu<\nu_{cr}$, then, 
by~\eqref{mean-var}, 
$$\frac{\E(T_{y(x),2})}{\E(T_{x,1})}\sim
\left(\frac{\nu_{cr}}{\nu}\right)^{\beta_2-1}>1.$$ 
Therefore, 
$(1+\delta)\E(T_{x,1})<(1-\delta)\E(T_{y(x),2})$
 for some  $\delta>0$ and sufficiently large $x$, and, hence, 
the  event
$$\left\{T_{x,1}<(1+\delta)\E(T_{x,1})\right\}\bigcap\left\{T_{y(x),2}>(1-\delta)\E(T_{y(x),2}
\right\}$$ 
implies the event $\{T_{x,1}<T_{y(x),2}\}$.
By~\eqref{cheb1},  
$$\P\left(T_{x,1}<(1+\delta)\E(T_{x,1})\right)
\geq 1-\P\left(|T_{x,1}-\E(T_{x,1})|\geq 
  \delta\E(T_{x,1}) \right)\geq 1-Cx^{-1}\to 1,$$
 as $x\to \infty$.
Similarly, 
$\P\left((1-\delta)\E(T_{y(x),2})<T_{y(x),2}\right)\to 1$, 
as $x\to\infty$.
Then, by independence of explosion times, 
\begin{equation*}
\P\left(T_{x,1}<T_{y(x),2}\right)\geq 
[\P\left(T_{x,1}<(1+\delta)\E(T_{x,1})\right)][\P\left(T_{y(x),2}>(1-\delta)
\E(T_{y(x),2})\right)]\to 1,
\end{equation*}
 as $x\to\infty$,
and, hence, $\lim_{\zeta\to\infty}\P_{(x,y(x))}\left(A_1\right)=1$,
as claimed.

If $y(x)=\nu x^{\alpha}+o(x^{\alpha})$ for $\alpha<\alpha_{cr}$,
then, $\alpha(1-\beta_2)-(1-\beta_1)>0$. By 
by equation~\eqref{mean-var} we obtain that   
$$\frac{\E(T_{y(x),2})}{\E(T_{x,1})}\sim  O(1)x^{\alpha(1-\beta_2)-(1-\beta_1)}
\to \infty,$$
as $x\to \infty$.
Therefore,     
$(1+\delta)\E(T_{x,1})<(1-\delta)\E(T_{y(x),2})$ for some $\delta>0$ and 
 sufficiently large $x$, and the claim of the theorem in this case follows
again from independence of the explosion times and bound~\eqref{cheb1}.
 
\subsection{Proof of Theorem~\ref{T1b}}
\label{proof-T1b}

We are going to use  notations introduced in Sections~\ref{prelim1} and~\ref{proof-T1a}.
\begin{proposition}
\label{P-on-curve}
If  $X(0)=(x, y(x))$, where 
$y(x)=\nu_{cr}x^{\alpha_{cr}}+\mu x^{\delta}+o(x^{\delta})$
for $\delta\in [0, \alpha_{cr})$,
 then 
$$\frac{\E(V_{x,y(x)})}{\sqrt{\Var(V_{x,y(x)})}}\sim -\rho x^{\delta-\frac{\alpha_{cr}}{2}},$$
where $\rho$ is defined in~\eqref{rho}.
 \end{proposition}
\begin{proof}
Using  equation~\eqref{mean-var} we obtain that
\begin{align*}
\E(V_{x,y(x)})&=\E(T_{x,1})-\E(T_{y(x),2})\sim \frac{x^{1-\beta_1}}{\beta_1-1}-
\frac{\left(\nu_{cr}x^{\alpha_{cr}}+
\mu x^{\delta}+o(x^{\delta})\right)^{1-\beta_2}}{\beta_2-1}
\\
&\sim\frac{x^{1-\beta_1}}{\beta_1-1}-\frac{\nu_{cr}^{1-\beta_2}x^{\alpha_{cr}(1-\beta_2)}}{\beta_2-1}
\left(1+\frac{\mu x^{\delta-\alpha_{cr}}}{\nu_{cr}}\right)^{1-\beta_2}\\
&\sim \frac{x^{1-\beta_1}}{\beta_1-1}-\frac{\nu_{cr}^{1-\beta_2}x^{\alpha_{cr}(1-\beta_2)}}{\beta_2-1}
\left(1+\frac{\mu(1-\beta_2)x^{\delta-\alpha_{cr}}}{\nu_{cr}}\right)\\
&\sim \frac{x^{1-\beta_1}}{\beta_1-1}-
\frac{\nu_{cr}^{1-\beta_2}x^{\alpha_{cr}(1-\beta_2)}}{\beta_2-1}
+\frac{\mu}{\nu_{cr}^{\beta_2}}x^{\delta-\alpha_{cr}\beta_2}.
\end{align*}
Noting that
$$\frac{x^{1-\beta_1}}{\beta_1-1}-
\frac{\nu_{cr}^{1-\beta_2}x^{\alpha_{cr}(1-\beta_2)}}{\beta_2-1}=0$$
we get the following approximation 
\begin{equation}
\label{Exp-Diff}
\E(V_{x,y(x)})=
\frac{\mu}{\nu_{cr}^{\beta_2}}x^{\delta-\alpha_{cr}\beta_2}.
\end{equation}
Further, by equation~\eqref{mean-var} we have that 
\begin{align*}
\Var(T_{x,1)})&\sim \frac{x^{-2\beta_1+1}}{2\beta_1-1}
\quad\text{and}\quad
\Var(T_{y(x),2})\sim 
\frac{\left(\nu_{cr}x^{\alpha_{cr}}\right)^{-2\beta_2+1}}
{2\beta_2-1}\left(1+\frac{\mu}{\nu_{cr}}x^{\delta-\alpha_{cr}}\right)^{-2\beta_2+1}.
\end{align*}
Observe that  
\begin{equation*}
\alpha_{cr}=\frac{\beta_1-1}{\beta_2-1}<\frac{m\beta_1-1}{m\beta_2-1}\quad \Longleftrightarrow
\quad
(-m\beta_1+1)-\alpha_{cr}(-m\beta_2+1)<0\quad\text{for}\quad m>1.
\end{equation*}
Therefore, 
 $
\Var(V_{x,y(x)})=\Var(T_{x,1})+\Var(T_{y(x),2})
\sim \Var(T_{y(x),2}),
$
and
\begin{equation}
\label{Var-Diff}
\sqrt{\Var(V_{x,y(x)})}
\sim \sqrt{\Var(T_{y(x),2})}\sim
 \frac{\nu_{cr}^{\frac{1}{2}-\beta_2}x^{\alpha_{cr}
\left(\frac{1}{2}-\beta_2\right)}}{\sqrt{2\beta_2-1}}.
 \end{equation}
Combining equations~\eqref{Exp-Diff} and \eqref{Var-Diff} gives the proposition.
\end{proof}

Further, observe that 
$\lim_{x\to\infty}\P\left(A^{\left(1\right)}_{x,y(x)}\right)=\lim_{x\to\infty}\P(V_{x,y(x)}<0)$
and
\begin{equation}
\label{L2}
\P(V_{x, y(x)}<0)=\P\left(\widetilde V_{x, y(x)}<-\frac{\E(V_{x, y(x)})}{\sqrt{\Var(V_{x, y(x)})}}\right)
\sim \P\left(\widetilde V_{x, y(x)}<\rho x^{\alpha-\frac{\alpha_{cr}}{2}}\right),
\end{equation}
where 
$\widetilde V_{x, y(x)}=\frac{V_{x, y(x)}-\E(V_{x, y(x)})}{\sqrt{\Var(V_{x, y(x)})}}.$
The normal approximation for explosion times 
discussed above yields the normal approximation for the random variable
$V_{x, y(x)}$.
In particular,   we have (similarly to~\eqref{Berry}), that
 \begin{equation}
 \label{norm2}
 |\P(\widetilde V_{x, y(x)}\leq z)-\Phi(z)|
 \leq O(1)x^{-\frac{\alpha_{cr}}{2}}\quad\text{for}\quad z\in \R.
 \end{equation}
Summarizing the above, we obtain the results listed below.
\begin{enumerate}
\item[i)] If $\alpha=\frac{\alpha_{cr}}{2}$, then 
 $\lim_{x\to\infty}\P\left(A^{\left(1\right)}_{x,y(x)}\right)=
  \Phi(\rho).$
\item[ii)]  If $\delta<\frac{\alpha_{cr}}{2}$, then 
$
 \lim_{x\to\infty}\P\left(A^{\left(1\right)}_{x,y(x)}\right)=
 \lim_{x\to\infty}\Phi\left(\rho x^{\delta-\frac{\alpha_{cr}}{2}}\right)
=\Phi(0)=\frac{1}{2}.$
\item[iii)] If $\frac{\alpha_{cr}}{2}<\delta<\alpha_{cr}$ 
\begin{enumerate}
\item and $\gamma<0$, then 
$\rho x^{\delta-\frac{\alpha_{cr}}{2}}\to \infty$, as $x\to \infty$, and, hence, 
$$
 \lim_{x\to\infty}\P\left(A^{\left(1\right)}_{x,y(x)}\right)=
 \lim_{x\to\infty}\Phi\left(\rho x^{\delta-\frac{\alpha_{cr}}{2}}\right)=1;$$
\item and $\gamma>0$, then 
$\rho x^{\delta-\frac{\alpha_{cr}}{2}}\to-\infty$, as $x\to \infty$, and, hence, 
$$
 \lim_{x\to\infty}\P\left(A^{\left(1\right)}_{x,y(x)}\right)=
 \lim_{x\to\infty}\Phi\left(\rho x^{\delta-\frac{\alpha_{cr}}{2}}\right)
=0.$$
\end{enumerate}
\end{enumerate} 
Theorem~\ref{T1b} now follows from items i)-iii).

\section{Proofs of results for the model with reflection}


In this section we prove a series of lemmas that 
describe  the long term behavior of the DTMC $\zeta(t)$  in the case, when 
the process starts  near either the lower, or   the upper boundary.  
The case of the lower  boundary is treated  in detail  in Section~\ref{lower}.
The case of the upper boundary is considered in Section~\ref{upper}.
The main results for the model with reflection, i.e. 
Theorems~\ref{main-sym} and~\ref{main-asym}, will follow 
from the series of  lemmas and explosiveness of the BB model with asymmetric feedback, see
 Section~\ref{proof-reflection-theorems} for details.

\subsection{The case of   the  lower boundary}
\label{lower}

Start with some definitions and auxiliary facts.

\begin{definition}
\label{definitions1}
{\rm 
Given $\beta_1, \beta_2: 1<\beta_1\leq \beta_2$ 
define 
\begin{equation}
\label{alpha-k}
\alpha_k=\frac{\beta_1-1}{\beta_2-1+\frac{1}{k}}\quad\text{for}\quad k\geq 1
\quad\text{and}\quad \alpha_0=0.
\end{equation}
}
\end{definition}
Note  that 
\begin{equation}
\label{alpha_k<alpha_k+1}
\alpha_k<\alpha_{k+1}\quad\text{and}\quad \lim\limits_{k\to \infty}\alpha_k=\alpha_{cr},
\end{equation}
where $\alpha_{cr}$ is defined in~\eqref{alpha-gamma}.
\begin{definition}
\label{stripes}
{\rm 
Let the lower boundary is given by the curve $y=x^{\alpha}$, for some 
$0<\alpha<\alpha_{cr}$.
Given $r>0$ define the set (stripe along the lower boundary)
\begin{equation}
\label{U}
U_{\alpha, r}=
\{(x, y)\in\N^2: x^{\alpha}\leq  y < x^{\alpha}+r\}.
\end{equation}
}
\end{definition}
If $0<\alpha<\alpha_{cr}$, then 
 $\alpha\beta_2-\beta_1<0$, which implies the following approximation 
\begin{equation}
\label{p-up}
\P_{\zeta}(\zeta_1(1)=x,\, \zeta_2(1)=y+1)\sim\frac{x^{\alpha\beta_2}}
{x^{\beta_1}+x^{\alpha\beta_2}}\sim x^{\alpha\beta_2-\beta_1}
\quad\text{for}\quad \zeta=(x,y): y\sim x^{\alpha}.
\end{equation}
Proposition~\ref{epsStrip1} below concerns a pair 
of  simple geometric   facts, that basically follow from sublinearity 
of the lower boundary $y=x^{\alpha}$ and the reflection rule defined in~\eqref{prob23}.
However, these  facts are repeatedly used in subsequent proofs, therefore it is convenient 
to arrange them in a separate  statement.

\begin{proposition}
\label{epsStrip1}
Consider  a trajectory of the DTMC $\zeta(t)$ that is reflected by the
 lower boundary  infinitely many times. 
\begin{enumerate}
\item[1)] 
Then, for such a trajectory the following holds
$$\liminf_{t\to\infty}(\zeta_2(t)-\zeta_1^{\alpha}(t))=0.$$
In other words, such a trajectory approaches the lower boundary at any arbitrarily small distance.
\item[2)] If, in addition, there are at most $m\geq 0$ jumps up 
between consecutive reflections, 
then such a trajectory  eventually confines to the strip $U_{\alpha, m+1+\eps}$
for any given $\eps>0$.
\end{enumerate}
\end{proposition}

\begin{proof}[Proof of Proposition~\ref{epsStrip1}]
Let $(\zeta_1, \zeta_2)$ be a state, at which  the DTMC
 $\zeta(t)$ is reflected by the lower boundary $y=x^{\alpha}$, i.e. the state satisfies~\eqref{prob23}
with $\varphi(x)=x^{\alpha}$.
Then, 
\begin{equation}
\label{delta(x)}
\zeta_2-\zeta_1^{\alpha}\leq \delta(\zeta_1):=(\zeta_1+1)^{\alpha}-\zeta_1^{\alpha}\sim 
\alpha (\zeta_1)^{\alpha-1}\to 0,\quad\text{as}\quad \zeta\to\infty.
\end{equation} 
 Let   $\tau_n$  be  the time, when
 the $n$-th reflection  takes place. It is obvious that $\tau_n\to \infty$, and, hence,
$\zeta_1(\tau_n)\to \infty$,  as $n\to \infty$.
Therefore, by~\eqref{delta(x)} 
\begin{equation}
\label{delta_n}
\zeta_2(\tau_n)-\zeta_1^{\alpha}(\tau_n)=\delta(\zeta_1(\tau_n))\to 0,\quad\text{as}\quad n\to \infty,
\end{equation} 
which implies  Part 1) of the proposition.

If, in addition,   there are at most $m\geq 0$ jumps up 
between consecutive reflections, then, by~\eqref{delta(x)}, 
 the process  does not exit  the strip $U_{\alpha, m+1+\delta(\zeta_1(\tau_n))}$
after $n$-th reflection. Combining this fact with~\eqref{delta_n} gives 
Part 2) of the proposition.
\end{proof}


Lemmas~\ref{LT4aa} and~\ref{Cor-A1} below describe in detail the long term behavior 
of the DTMC $\zeta(t)$ in the case, when the parameter $\alpha$ 
(determining the lower boundary $y=x^{\alpha}$) is smaller than $\alpha_{1}=\frac{\beta_1-1}{\beta_2}$, and the process starts near  the lower boundary.

\begin{lemma}
\label{LT4aa}
Let   $\alpha\in (0, \alpha_1)$, and let $\zeta=(x,y)\in U_{\alpha, C}$ for some $C>0$.
Then
$$\P_{\zeta}(\{\text{the DTMC }
 \zeta(t) \text{  jumps to the right, whenever  possible}\})\to 1,\quad\text{as}\quad x\to \infty.$$
\end{lemma}
\begin{proof}
[Proof of Lemma~\ref{LT4aa}]
 Denote  for short
$$A=\{\text{the DTMC }
 \zeta(t) \text{  jumps to the right, whenever  possible}\}.$$
It is easy to see that, given an initial state $\zeta$, the event $A$
consists of a single trajectory, such that the  component $\zeta_2(t)$
 increases  only at those moments in time, when the process is reflected by the lower boundary.
It is easy to see that, since the function $y=x^{\alpha}$ is monotonically  increasing to infinity,
the trajectory  does not leave the strip $U_{\alpha, C}$ before the first reflection, and then, by Proposition~\eqref{epsStrip1},
 it confines to the strip $U_{\alpha, 2}$. 
By~\eqref{p-up}   
\begin{align*}
\P_{\zeta}(A)&\sim \prod\limits_{k=x}^{\infty}
\left(1-k^{\alpha\beta_2-\beta_1}\right)\sim 
e^{-\sum_{k=x}^{\infty}k^{\alpha\beta_2-\beta_1}},\quad\text{as}\quad x\to\infty,
\quad\text{for}\quad \zeta=(x,y)\in U_{\alpha, 2}.
\end{align*}
Since $\alpha<\alpha_1=\frac{\beta_1-1}{\beta_2}\iff  \alpha\beta_2-\beta_1<-1$, we get 
 that 
$\sum_{k=x}^{\infty}k^{\alpha\beta_2-\beta_1}\to 0$, as $x\to\infty$. Therefore, 
$\P_{\zeta}(A)\to 1$, as $x\to \infty$, and the lemma is proved.
\end{proof}

\begin{lemma}
\label{Cor-A1}
Let   $\alpha\in (0, \alpha_1)$, and let $\zeta=(x,y)\in U_{\alpha, C}$ for some $C>0$. Then
$$\P_{\zeta}(A_{\alpha, 1})\to 1, \quad\text{as}\quad x\to\infty,$$
where  $A_{\alpha, 1}$ is the event defined in~\eqref{Ak}.
\end{lemma}
\begin{proof}[Proof of Lemma~\ref{Cor-A1}]
Consider the same trajectory as in the proof of Lemma~\ref{LT4aa}, i.e. where 
the  component $\zeta_2(t)$
 increases  only due to reflections at the lower boundary.
It is easy to see that this trajectory is reflected by the lower boundary infinitely many times.
Consequently, by  Part 1) of Proposition~\ref{epsStrip1}, we have that 
$\liminf_{t\to\infty}(\zeta_2(t)-\zeta_1^{\alpha}(t))=0$.
The component $\zeta_2(t)$ increases by $1$ at the moment of reflection, therefore, 
by Part 2) of  Proposition~\ref{epsStrip1},  the trajectory eventually  confines 
to the strip $U_{\alpha, 1+\eps}$ for any given $\eps>0$, so that 
$\limsup_{t\to\infty}(\zeta_2(t)-\zeta_1^{\alpha}(t))=1.$
Lemma~\ref{Cor-A1} follows now from Lemma~\ref{LT4aa}.
\end{proof}

To proceed further, recall the second item in  Definition~\ref{jumps-reflections} that distinguishes 
between   jumps and reflections 
of the DTMC $\zeta(t)$.

\begin{lemma}
\label{PR-binom}
Given $\alpha\in (0, \alpha_{cr})$,  let  $\zeta=(x, y)\in U_{\alpha, C_1}$ and let  
$N=N(x)=C_2x^{1-\alpha}$ for some  $C_1>0$ and $C_2>0$.
Let $\xi_N$ be the number of jumps up 
of the DTMC $\zeta(t)$ on the interval $[0, N]$. 
Suppose that $x$ is sufficiently large.
Then for any fixed $m\geq 1$ the following bounds hold
\begin{equation}
\label{Poisson}
C_3\lambda_x^{m}\leq \P_{\zeta}(\xi_N\geq m)\leq 
C_4\lambda_x^{m},
\end{equation}
where 
\begin{equation}
\label{lambda-n}
\lambda_x:=x^{\alpha\beta_2-\beta_1+1-\alpha},
\end{equation}
and constants $C_3>0$ and $C_4>0$  do not depend on $x$.
\end{lemma}
\begin{proof}[Proof of Lemma~\ref{PR-binom}]
Start with noting that the condition 
$\alpha<\alpha_{cr}$ is equivalent to the condition $\alpha\beta_2-\beta_1+1-\alpha<0$, which, 
in particular,  implies that $\lambda_x\to 0$,  as $x\to \infty$.
Also, in this case we have that 
$\alpha\beta_2-\beta_1<0$.
Therefore, 
by equation~\eqref{p-up}, 
\begin{equation}
\label{p-up-1}
\P_{\zeta}(\zeta_1(1)=x,\, \zeta_2(1)=y+1)\sim x^{\alpha\beta_2-\beta_1}.
\end{equation}
Next,  assuming that the process does not jump up 
more than a fixed number of times  (not depending on $x$)  on the interval $[0,N]$, we 
can approximate  the probability of the jump up at any time $t\in [0,N]$ by the probability of the jump 
up at time $0$, that is, $x^{\alpha\beta_2-\beta_1}$ (see~\eqref{p-up-1}) for sufficiently large $x$.
This allows to approximate the distribution of the random variable $\xi_N$ by the binomial distribution
with the  mean $\lambda_x$, which, in turn, can be approximated by the Poisson 
distribution with the same mean.
More precisely,  noting first  that 
$$
(1-x^{\alpha\beta_2-\beta_1})^N\leq 
\P_{\zeta}(\xi_N=0)\leq (1-(x+N)^{\alpha\beta_2-\beta_1})^N,
$$
and 
$$(1-x^{\alpha\beta_2-\beta_1})^N\sim (1-(x+N)^{\alpha\beta_2-\beta_1})^N
\sim e^{-\lambda_x},
$$
we obtain that 
$$\P_{\zeta}(\xi_N=0)\sim e^{-\lambda_x},\quad\text{as}\quad x\to \infty.$$
Combining this with equation~\eqref{p-up-1} gives  that
$$\P_{\zeta}(\xi_N=i)\sim e^{-\lambda_x}\frac{ \lambda_x^i}{i!},\quad\text{as}\quad x\to \infty,$$
  for any fixed $i$, which, in turn, implies the bound~\eqref{Poisson}, as claimed.
\end{proof}

\begin{definition}
\label{t-S}
{\rm 
Let $\alpha\in (0, \alpha_{cr})$ be the parameter determining the lower boundary $y=x^{\alpha}$.
Given $r>0$ and   
 $\zeta=(x,y): x^{\alpha}\leq y\leq x^{\gamma}$, 
define $n_x$ as an integer such that 
\begin{equation}
\label{n0}
\left(n_xr\right)^{\frac{1}{\alpha}}\leq x<\left((n_x+1)r\right)^{\frac{1}{\alpha}}
\end{equation}
and 
$$t_{r,n}=\min(t: \zeta_1(t_{r,n} )\geq \left(nr\right)^{\frac{1}{\alpha}})\quad\text{for}\quad n>n_x.$$
}
\end{definition}

\begin{remark}
{\rm 
Note that due to  reflection 
 both $\zeta_1(t)\to \infty$, and 
$\zeta_2(t)\to \infty$, as $t\to \infty$.  Therefore, 
$t_{r,n}<\infty$   and $t_{r,n}<t_{r, n+1}$ for all $n$.
In addition, note that 
\begin{equation}
\label{x_n-growth}
\left((n+1)r\right)^{\frac{1}{\alpha}}-\left(nr\right)^{\frac{1}{\alpha}}
\sim 
\frac{r}{\alpha}\left(nr\right)^{\frac{1-\alpha}{\alpha}},\quad\text{as}\quad n\to \infty.
\end{equation}
}
\end{remark}
\begin{definition}
\label{Srn}
{\rm 
Let $(t_{r,n},\, n>n_x)$ be the sequence of time moments defined in Definition~\ref{t-S}
for some $\zeta=(x,y)$.
Define 
$S_{r, n}$ as the random variable that is equal to the number of jumps up 
of the  DTMC $\zeta(t)$ on the interval $[t_{r, n}, t_{r, n+1})$. 
}
\end{definition}


\begin{lemma}
\label{2k-1+eps}
Let    $\alpha \in [\alpha_{k-1},  \alpha_k)$ for some  $k\geq 2$,
and  let $\zeta=(x,y)\in U_{\alpha, C}$ for some $C>0$.
Let $(S_{r,n},\, n>n_x)$ be the sequence 
of random variables defined in Definition~\ref{Srn}.
Then,   
\begin{equation}
\P_{\zeta}\left(S_{r, n}<k \text{ for all } n>n_x\right)\to 1,\quad\text{as}\quad x\to \infty,\label{Brk}
\end{equation}
and 
\begin{equation}
\P_{\zeta}\left(S_{r, n}=k-1 \text{ for infinitely many } n\right)=1.
\label{Brk1}
\end{equation}

\end{lemma}

\begin{proof}[Proof of Lemma~\ref{2k-1+eps}]

Start with showing~\eqref{Brk} in the case, when $r$ is sufficiently large.
Recall that $n_0$ is a positive integer, such that~\eqref{n0} holds. 
It is easy to see  that 
the lower boundary  $y=x^{\alpha}$ increases  by $r$ on the interval 
$[\left(rn\right)^{\frac{1}{\alpha}}, \left(r(n+1)\right)^{\frac{1}{\alpha}})$.
Therefore, if $r$ is sufficiently large, and 
$S_{r, n}<k$ for all $n>n_x$, then 
\begin{equation}
\label{Uk}
\{S_{r, n}<k\,\text{  for all }\,  n\}\subseteq \{\zeta(t)\in U_{\alpha, 2k-1+C}\,\text{ for all }\, t\geq 0\},
\end{equation}
i.e. the process cannot escape the strip $U_{\alpha, 2k-1+C}$. 
Indeed, it is easy to see that the difference $\zeta_2(t)-\zeta_1^{\alpha}(t)$ cannot be 
more than $2k-1+C$. On the other hand, 
 the difference $\zeta_2(t)-\zeta_1^{\alpha}(t)$ can be equal to 
$2k-1+C$.  Indeed, this can happen, 
when  the process is reflected near the end of the  interval $[t_{r, n}, t_{r, n+1})$, immediately 
jumps  up $k-1$ times, and,  jumps up again  $k-1$ times at the beginning 
of the next interval.

Further, by Definition~\ref{t-S},  
\begin{equation}
\label{Uk1}
\zeta_1(t_{r,n})-\left(rn\right)^{\frac{1}{\alpha}}\leq 1\quad\text{for all}\quad n>n_x.
\end{equation}
Without loss of generality, we  can also assume that the initial position $\zeta=(x,y)$ 
 is such that
\begin{equation}
\label{Uk10}
x-\left(rn_x\right)^{\frac{1}{\alpha}}\leq 1.
\end{equation}
Then, it follows from~\eqref{Uk}-\eqref{Uk10}  and 
Lemma~\ref{PR-binom} that
  \begin{equation}
\label{Pk1}
\P_{\zeta}(S_{r,n}\geq k|S_{r,n_x}<k,..., S_{r, n-1}<k)\leq C_3
\lambda_n^{k},
\end{equation}
and, hence,  
\begin{equation}
\label{a1}
\P_{\zeta}(\cap_{i=n_x}^n\{S_{r,i}<k\})
\geq  \prod_{i=n_x}^{n}(1-C_3\lambda_i^k),
\end{equation}
where 
\begin{equation}
\label{lambda-n}
\lambda_i:=(ir)^{\frac{\alpha(\beta_2-1)-(\beta_1-1)}{\alpha}}\quad\text{for}\quad i\geq 0.
\end{equation}
Note that the assumption 
$\alpha\in[\alpha_{k-1}, \alpha_k)$ can be rewritten as follows
\begin{equation}
\label{k_k-1}
\frac{\alpha(\beta_2-1)-(\beta_1-1)}{\alpha}k<-1\leq \frac{\alpha(\beta_2-1)-(\beta_1-1)}{\alpha}(k-1).
\end{equation}
The left inequality in the preceding display implies  that 
$\sum_{n=n_x}^{\infty}\lambda_n^k\to 0$, as $n_x\to \infty$. Combining this 
with~\eqref{a1} gives that
\begin{equation}
\begin{split}
\label{a2}
\P_{\zeta}\left(S_{r, n}<k \text{ for all } n\right)&
\geq \prod\limits_{n=n_x}^{\infty}\left(1-C_3\lambda_n^k\right)
\geq e^{-O(1)\sum\limits_{n=n_x}^{\infty}\lambda_n^k}
\to 1,\quad\text{as}\quad x\to\infty.
\end{split}
\end{equation}

Next, consider the case of an arbitrary $r$. 
It is easy to see that in this case  time intervals  $[t_{r, n},  t_{r, n+1})$ ($r$-intervals)
can be covered by  similar intervals with  a sufficiently   large  $r'>r$ ($r'$-intervals) in such a manner 
that each $r$-interval  is covered by an $r'$-interval.
Clearly, if it occurs that $S_{r, n}\geq k$  for infinitely many  $n$,
 then  there are infinitely many  $r'$-intervals, such that 
$S_{r', n}\geq k$. 
The probability of the latter tends to zero, as we have proved in the case 
of large $r$, and, hence,~\eqref{BrK} holds for an arbitrary $r$.

Next, let us show~\eqref{Brk1}.
To this end observe that 
  \begin{equation}
\label{Pk2}
\P_{\zeta}(S_{r,n}\geq k-1|{\cal F}_{n-1})\geq 
 C_2 \lambda_n^{k-1},
\end{equation}
where ${\cal F}_{n-1}$ is  the $\sigma$-field generated by $\zeta(t),\, t<t_{r,n}$.
Indeed, it suffices  to consider 
the  case  when the process starts near the lower boundary
(i.e. $\zeta\in U_{\alpha, C}$ for some $C>0$).
By  Lemma~\ref{PR-binom} 
and  in~\eqref{k_k-1}, we have that  
$\sum_{n=1}^{\infty}\lambda_n^{k-1}=\infty$,
and~\eqref{Brk1} follows from the conditional variant of the second Borel-Cantelli lemma 
(e.g. see Theorem 5.3.2 in~\cite{Durrett}).
\end{proof}

\begin{lemma}
\label{0+k}

Let    $\alpha \in [\alpha_{k-1},  \alpha_k)$ for some  $k\geq 2$,
and  let $\zeta=(x,y)\in U_{\alpha, C}$ for some $C>0$.
Then 
$\P_{\zeta}\left(A_{\alpha, k}\right)\to 1$,
as $x\to\infty,$ 
where $A_{\alpha, k}$ is the event defined in~\eqref{Ak}.
\end{lemma}

\begin{proof}[Proof of Lemma~\ref{0+k}]
We  use notations of Lemma~\ref{2k-1+eps}.

First, given a sufficiently large $r>0$ consider a trajectory of the process, such 
that  $S_{r, n}<k$  for all $n$. Comparing the  growth of the trajectory with that of 
the lower boundary (similarly to the proof of Lemma~\ref{2k-1+eps})
gives  that such a trajectory is reflected by the lower boundary 
at least once on each time interval $[t_{r, n}, t_{r, n+1})$.  
Consequently, the trajectory is reflected infinitely many times. Combining this with 
 Part 1) of  Proposition~\ref{epsStrip1} and Lemma~\ref{2k-1+eps} gives  that
\begin{equation}
\label{liminf=0}
\lim_{x\to\infty}\P_{\zeta}\left(\liminf_{t\to \infty}(\zeta_2(t)-\zeta_1^{\alpha}(t))=0\right)=1\quad
\text{for}\quad \zeta=(x,y)\in  U_{\alpha, C}. 
\end{equation}
Further, given  $n\geq 1$ let  $\tau_n$ be the  time moment, when the process is reflected for 
 the $n$-th time, and let $\tau_0=0$.
Let $S_{n}$ be the number of jumps up on the interval $[\tau_n, \tau_{n+1})$ for $n\geq 0$.
It is easy to see that, if $S_n<k$ for all $n$, then 
differences $\tau_{n+1}-\tau_{n}$ are uniformly bounded, provided that 
$S_n<k$ for all $n$. Therefore, any of these  intervals is covered by some 
 deterministic  interval $[t_{r,n}, t_{r, n+1})$ for a sufficiently  large $r$. 
Consequently, with a little of extra work,  it follows 
from Part 1) of Lemma~\ref{2k-1+eps} 
that 
$$\lim\limits_{x\to\infty}\P_{\zeta}(S_n<k\text{ for all } n)=1,\quad\text{for}
\quad  \zeta=(x,y)\in U_{\alpha, C}.$$
Combining this result with Part 2) of 
 Proposition~\ref{epsStrip1}, gives that, with  $\P_{\zeta}$-probability
  tending to $1$, as $x\to\infty$, 
\begin{itemize}
\item[a)]  the DTMC $\zeta(t)$ eventually confines
 to the strip $U_{\alpha, k+\eps}$ for any given $\eps>0$.
\end{itemize}

On the other hand,
\begin{itemize}
\item[b)]  for any given $\eps>0$
the  DTMC $\zeta(t)$  almost surely exits the strip $U_{\alpha, k-\eps}$ infinitely many times.
\end{itemize}
Indeed, choose  $r>0$ in Part 2) of Lemma~\ref{2k-1+eps} is such that 
the lower boundary $y=x^{\alpha}$ grows between points $(nr)^{\frac{1}{\alpha}}$ and  
$((n+1)r)^{\frac{1}{\alpha}}$ by not more than an arbitrarily small $\eps'>0$.
By Part 2) of Lemma~\ref{2k-1+eps}, almost surely there are infinitely many 
intervals $t_{r,n}, t_{r, n+1})$, where the process jumps up exactly $k-1$ times. This gives b).

Items a) and b) imply that
\begin{equation}
\label{limsup=k}
\lim_{x\to\infty}\P_{\zeta}\left(\limsup_{t\to \infty}(\zeta_2(t)-\zeta_1^{\alpha}(t))=k\right)=1
\quad\text{for}\quad \zeta=(x,y)\in U_{\alpha, C}.
\end{equation}

The lemma now follows from~\eqref{liminf=0} and~\eqref{limsup=k}.
\end{proof}

\subsection{The case of the upper boundary}
\label{upper}

In this section we study the long term behavior of the DTMC $\zeta(t)$ in the case, when the 
process  starts near the upper boundary. Recall that $1<\beta_1\leq \beta_2$, and the upper 
boundary is given by the curve 
$y=x^{\gamma}$, where  $\gamma>\alpha_{cr}=\frac{\beta_1-1}{\beta_2-1}$.
We distinguish three cases,  namely, $\gamma>1$
(superlinear upper boundary), $\gamma=1$ (the case of the bisector) and 
$\alpha_{cr}<\gamma<1$ (sublinear upper boundary). 
Note that the last two cases are possible only in the asymmetric case $1<\beta_1<\beta_2$.
The case of the super-linear upper boundary 
and the case of the bisector are considered in Section~\ref{Up-super}, and
the case of  the sublinear upper boundary is considered  in
Section~\ref{Up-sublinear}.

\subsubsection{Upper boundary: the bisector and above}
\label{Up-super}

In this section we study the most probable 
 long term behavior of the DTMC $\zeta(t)$ in the case when it starts
near the upper boundary given by $y=x^{\gamma}$, where $\gamma\geq 1$.
In this case the corresponding results are similar to the results obtained in Section~\ref{lower}
in the case of the lower boundary. 
Therefore, we  only state the results and relate  them to their analogues in Section~\ref{lower}.
However, in the case of a bisector 
(which is possible only in the asymmetric case $1<\beta_1<\beta_2$),
 the corresponding results can be refined (see Remark~\ref{Cor1-bisector}
and Lemma~\ref{Cor2-bisector} below).

We start with some definitions.
\begin{definition}
\label{definitions2}
{\rm 
Given $\beta_1, \beta_2: 1<\beta_1\leq \beta_2$ 
define 
\begin{equation}
\label{gamma-k}
\gamma_k=\frac{\beta_1-1+\frac{1}{k}}{\beta_2-1}=\alpha_{cr}+\frac{1}{(\beta_2-1)k}\quad\text{for}
\quad k\in \N.
\end{equation}
}
\end{definition}

\begin{remark}
{\rm 
Note that quantities $(\gamma_k,\, k\geq 1)$ are similar  to quantities $(\alpha_k,\, k\geq 1)$ (defined
 in~\eqref{alpha-k}).
Note that 
$\gamma_{k+1}<\gamma_{k}$ for any $k\geq 1$.  If 
 $\beta_1=\beta_2$, then  
$\gamma_k=\alpha_k^{-1}$ for any $k\geq 1$.
}
\end{remark}

\begin{definition}
\label{tilde-stripe}
{\rm 
Let the upper boundary given by the curve $y=x^{\gamma}$, where $\gamma>1$.
Given $r>0$
define the set (stripe along the upper boundary)
\begin{equation}
\label{tilde-U}
\begin{split}
\widetilde{U}_{\gamma, r}&=
\{(x,y)\in\N^2:
x\geq r,\, (x-r)^{\gamma} < y \leq  x^{\gamma}\}\\
&=
\{(x, y)\in\N^2:x\geq r,\,
y^{\frac{1}{\gamma}}\leq  x < y^{\frac{1}{\gamma}}+r\}.
\end{split}
\end{equation}
}
\end{definition}

\begin{remark}
{\rm 
The stripes $\widetilde{U}_{\gamma, r}$  are analogous of  stripes 
$U_{\alpha, r}$  along  the lower  boundary defined in Definition~\ref{stripes}.
}
\end{remark}

Observe, that in the case, when  $\frac{\beta_1}{\gamma}-\beta_2<0$, we have the following approximation
for the probability of the jump to the right (see~\eqref{prob21-22})
\begin{equation}
\label{p-right}
\P_{\zeta}(\zeta_1(1)=x+1,\, \zeta_2(1)=y)
\sim\frac{y^{\frac{\beta_1}{\gamma}}}
{y^{\frac{\beta_1}{\gamma}}+y^{\beta_2}}\sim y^{\frac{\beta_1}{\gamma}-\beta_2}
\sim x^{\beta_1-\gamma\beta_2}
\quad\text{for}\quad \zeta=(x,y): y\sim x^{\gamma}.
\end{equation}

\begin{lemma}
\label{gamma-1}
Recall that $1<\beta_1\leq \beta_2$ and suppose one of the following two conditions is fulfilled
\begin{itemize}
\item[a)]  $\beta_2-\beta_1\leq 1$ (i.e. $\gamma_1=\frac{\beta_1}{\beta_2-1}\geq 1$ )
and $\gamma>\max(1, \gamma_1)$;  
\item[b)]  $\beta_2-\beta_1>1$  (i.e. $\gamma_1=\frac{\beta_1}{\beta_2-1}<1$) and $\gamma\geq 1$
\end{itemize}
and $\zeta=(x,y)\in \widetilde{U}_{\gamma, C}$ for some $C>0$.
Then, 
\begin{enumerate}
\item[1)] $\lim_{y\to\infty}\P_{\zeta}(\text{the DTMC }  \zeta(t)   \text{ jumps up, whenever possible})=1$,
\item[2)] $\lim_{y\to\infty}\P_{\zeta}(\widetilde{A}_{\gamma, 1})=1$,
where the event $\widetilde{A}_{\gamma, 1}$ is defined in~\eqref{Ak-tilde}.
\end{enumerate}
\end{lemma}
\begin{proof}
Part 1) of Lemma~\ref{gamma-1} is an analogue of  Lemma~\ref{LT4aa}.
By~\eqref{p-right}, 
 the probability that the process jumps up whenever  possible can be estimated 
by 
$e^{-\sum_{k=y}^{\infty}k^{\frac{\beta_1}{\gamma}-\beta_2}}$.
Both a) and b) yield, that  $\frac{\beta_1}{\gamma}-\beta_2<-1$.
Therefore,  the expression 
in the exponent above is a tail of a convergent series, so that it tends to $0$, and, hence, the exponent 
 tends to $1$, as $y\to\infty$.

  Part 2) of Lemma~\ref{gamma-1}  is the  analogue  of 
 Lemma~\ref{Cor-A1} and  can be proven similarly. 
\end{proof}

\begin{remark}
\label{Cor1-bisector}
{\rm 
Note a special case of Lemma~\ref{gamma-1}, when
 $\beta_2-\beta_1>1$ and $\gamma=1$ (i.e. the upper boundary is given by the bisector). 
In this case,  
the process started at a state $(x,x)$  follows  the "zigzag" trajectory 
$$\zeta=(x, x)\to (x+1, x)\to (x+1, x+1)\to (x+2, x+1)\to (x+2, x+2)\to....$$
with $\P_{\zeta}$ probability tending to $1$, as $x\to\infty$.
}
\end{remark}

\begin{definition}
\label{t-S-up}
{\rm 
Let $\gamma\geq 1$ be the parameter determining the lower boundary $y=x^{\gamma}$.
Given $r>0$ and   
 $\zeta=(x,y)$ (belonging to the state space of the process) 
define $n_y$ as an integer such that 
\begin{equation}
\label{n0}
\left(n_yr\right)^{\gamma}\leq y<\left((n_y+1)r\right)^{\gamma}
\end{equation}
and 
$$\tilde t_{r,n}=
\min(t: \zeta_2(t_{r,n} )\geq \left(nr\right)^{\gamma})\quad\text{for}\quad n>n_y.$$
}
\end{definition}
\begin{definition}
\label{Srn-up}
{\rm 
Let $(\tilde t_{r,n},\, n>n_y)$ be the sequence of time moments defined in Definition~\ref{t-S-up}
for some $\zeta=(x,y)$.
Define 
$\widetilde S_{r, n}$ as the random variable that is equal to the number of jumps to the right 
of the  DTMC $\zeta(t)$ on the interval $[\tilde t_{r, n}, \tilde t_{r, n+1})$. 
}
\end{definition}

\begin{remark}
{\rm 
Note that if $\beta_2-\beta_1<1$, then 
$\gamma_k\geq 1$ only for  $k=1,...,k_{max}$, where 
 $k_{max}$ is the  maximal integer, such that 
$k\leq \frac{1}{\beta_2-\beta_1}$.
}
\end{remark}

\begin{lemma}
\label{2k-1+eps-upper}
Let $\beta_2-\beta_1<1$ and $\max(1, \gamma_k)<\gamma\leq \gamma_{k-1}$ for some 
integer $k\geq 2$, 
and let $\zeta=(x,y)\in \widetilde{U}_{\gamma, C}$ for some $C>0$.
Then,  for any given  $r>0$, 
\begin{equation}
\lim\limits_{y\to\infty}\P_{\zeta}\left(\widetilde{S}_{r, n}<k
 \text{ for all } n\right)=1
\label{Brk-up}
\end{equation}
and 
\begin{equation}
\P_{\zeta}\left(\widetilde{S}_{r, n}=k-1 \text{ for infinitely many } n\right)=1.
\label{Brk1-up}
\end{equation}
\end{lemma}
Lemma~\ref{2k-1+eps-upper} is an analogue of Lemma~\ref{2k-1+eps} and can be proven 
in a similar way (by using an analogue of Lemma~\ref{PR-binom}). 
\begin{remark}
\label{bi}
{\rm Note that in the special case $\gamma=1$ in 
Lemma~\ref{2k-1+eps-upper}  we have that 
 $((n+1)r)^{\gamma}-(nr)^{\gamma}=(n+1)r-nr=r$, which implies that the intervals 
$[t_{r,n}, t_{r,n+1})$ are uniformly bounded by a constant. In turn, this implies 
that 
$$\P_{\zeta}\left(\widetilde{S}_{r, n}\geq k|
\widetilde{S}_{r, n_y}<k,...,\widetilde{S}_{r, n-1}<k
\right)\sim O(1)n^{(\beta_1-\beta_2)k}$$
for any possible integer $k\geq 0$.
Therefore, if $\frac{1}{k}<\beta_2-\beta_1\leq \frac{1}{k-1}$ (i.e. 
$k-1\leq \frac{1}{\beta_2-\beta_1}<k$), then 
 $\sum_n n^{(\beta_1-\beta_2)k}<\infty$ and  $\sum_n n^{(\beta_1-\beta_2)(k-1)}=\infty$, which 
implies~\eqref{Brk-up} and~\eqref{Brk1-up} respectively. 
}
\end{remark}
\begin{lemma}
\label{gamma-k}
Let $\beta_2-\beta_1<1$ and $\max(1, \gamma_k)<\gamma\leq \gamma_{k-1}$ for some $k\geq 2$ and 
let  $\zeta=(x,y)\in \widetilde{U}_{\gamma, C}$ for some $C>0$. Then 
$\lim_{y\to\infty}\P_{\zeta}(\widetilde{A}_{\gamma, k})=1$, 
where the event  $\widetilde{A}_{\gamma, k}$ is  defined in~\eqref{Ak-tilde}.
\end{lemma}
Lemma~\ref{gamma-k} is an analogue of Lemma~\ref{0+k} and can be proven similarly.

Lemma~\ref{Cor2-bisector} below  is a refinement (implied by Remark~\ref{bi}) of Lemma~\ref{gamma-k}
in the case,   when  the upper boundary  is  the bisector.

\begin{lemma}
\label{Cor2-bisector}
 Let $\gamma=1$ and let $\zeta=(x, x)$ for some $x\in \N$.
If $\frac{1}{k}<\beta_2-\beta_1\leq \frac{1}{k-1}$ (i.e. 
$k-1\leq \frac{1}{\beta_2-\beta_1}<k$)
for some integer $k\geq 2$, then
$$\lim_{x\to\infty}\P_{\zeta}(D_1\cap D_2\cap D_3)=1,$$
where events $D_i,\, i=1,2,3$ are  defined below
\begin{align*}
D_1&=\{\zeta_1(t)-\zeta_2(t)\leq k\text{  for all  } t\geq 0\},\\
D_2&=\{\zeta_1(t)-\zeta_2(t)=k\text{  for infinitely many  } t\},\\
D_3&=\{\zeta_1(t)=\zeta_2(t)\text{ for infinitely  many  } t\}.
\end{align*}
\end{lemma}

\subsubsection{Upper boundary:  below the bisector and above the critical curve}
\label{Up-sublinear}

Define 
$$W_{\gamma, C}=\{\zeta=(x,y)\in \N^2: x^{\gamma}-C<
y\leq x^{\gamma}\}\quad\text{for}\quad C>0.$$

\begin{lemma}
\label{LT4ba}
Let $1<\beta_1<\beta_2$ and 
 $\gamma\in(\alpha_{cr}, 1)$,  and let 
$\zeta=(x,y)\in W_{\gamma, C}$ for some $C>0$.
Then, 
$\lim_{x\to\infty}\P_{\zeta}(B_{\gamma})=1$,
where the event $B_{\gamma}$ is defined in~\eqref{Bgamma}.
\end{lemma}

\begin{proof}[Proof of Lemma~\ref{LT4ba}] 

Without loss of generality, assume that the initial 
state  $\zeta=(x,y)\in W_{\gamma, 1}$, in which case
$y=x^{\gamma}-1+\eps_0$ for some $0<\eps_0<1$. 
Given  $0<\eps<1$, define the sequence of stopping times
\begin{align*}
t_{0}&=\min(t: \zeta_1^{\gamma}(t)\geq x^{\gamma}+\eps_0),\\
t_{n}&=\min(t: \zeta_1^{\gamma}(t)\geq x^{\gamma}+\eps+n)\text{ for }\,n\geq 1,\\
t'_{n}&=\min(\zeta_1^{\gamma}(t)\geq x^{\gamma}+2\eps+n)\,\text{ for }\,n\geq 1.
\end{align*}
Note that 
 $0\leq t_0<t_1'<t_1<...<t'_n<t_n<t'_{n+1}<t_{n+1}<...$.
It follows from equation~\eqref{prob24} and the choice of the initial condition, that   the 
DTMC $\zeta(t)$  is reflected by the upper boundary all the time in the interval $t\in [0, t_0)$.
Also, the process can jump up only {\it once} in the interval $[t_0, t'_1)$. 
We show that, with $\P_{\zeta}$ probability tending to $1$, as $x\to \infty$, 
the  process behaves  in the same manner during  the subsequent time intervals.
Namely, it   jumps up  on each  time interval  $[t_n, t'_{n+1})$, and, as a result, 
it is   reflected by the upper boundary all the time on each time interval 
 $t\in [t'_{n}, t_{n})$.
Clearly,  such a behavior implies that  the process stays inside the strip $W_{\gamma, 1+\eps}$ 
for all $t\geq t'_1$. 
It is convenient to introduce random variables $S_n,\, n\geq 0$, where 
 $S_{n}$ is the number of  jumps up of the DTMC $\zeta(t)$ 
in the interval $[t_n, t'_{n+1})$. In terms of these random variables
it suffices  to show that 
\begin{equation}
\label{S_n=1}
\lim_{\zeta\to\infty}\P_{\zeta}(S_n=1\, \text{ for all } \, n\geq 0)=1,
\end{equation}
to guarantee the described above behavior of the process.

To proceed, observe first that, under our assumptions on the initial condition,
if $S_0=1, ..., S_{n-1}=1$, then 
\begin{align}
\label{trn1}
N_n&:=t_{n+1}'-t_{n}\sim 
O(1)\left(x^{\gamma}+n\right)^{(1-\gamma)/\gamma},\\
\label{trn2}
\zeta_{1, n}&:=\zeta_1(t_n)\sim t_n+x\sim O(1)\left(x^{\gamma}+n\right)^{1/\gamma},
\end{align}
as $x, n\to\infty$.

Next, consider two cases.
\paragraph{{\it Case 1):} $\gamma\in (\alpha_{cr}, \frac{\beta_1}{\beta_2}).$}

In this case $\gamma\beta_2-\beta_1<0$, so that  
$$
\P_{\zeta}(\zeta_1(1)=x+1,\, \zeta_2(1)=y)\sim 1-x^{\gamma\beta_2-\beta_1}
\text{  for }\,
\zeta=(x,y): y\sim x^\gamma, \text{  as }\, x\to\infty.
$$
The above gives that 
\begin{equation}
\label{mb0}
\P_{\zeta}(S_n=0|S_0=1, ..., S_{n-1}=1)
\sim  \left(1-\zeta_{1,n}^{\gamma\beta_2-\beta_1}\right)^{N_n}\sim 
e^{-\zeta_{1,n}^{\gamma\beta_2-\beta_1} N_n},
\end{equation}
for $n\geq 0$.

By~\eqref{trn1}, ~\eqref{trn2}and~\eqref{mb0},
\begin{equation}
\label{mb1}
\P_{\zeta}(S_n=1|S_0=1, ..., S_{n-1}=1)
\sim 
1-e^{-a(x^{\gamma}+\eps+n)^{b }},
\end{equation}
where $a=2\eps/\gamma$ and 
$b=\frac{1-\gamma+\gamma\beta_2-\beta_1}{\gamma} >0,$
and, hence, 
\begin{equation*}
\label{lim1}
\begin{split}
\lim\limits_{\zeta\to\infty}\P_{\zeta}\left(S_{n}=1\, \text{ for all }\, n\geq 0\right)
&=
\lim\limits_{x\to\infty}\prod\limits_{n=0}^{\infty}
\left(1-e^{-a(x^{\gamma}+\eps+n)^{b}}\right)=1,
\end{split}
\end{equation*}
as required.

\paragraph{{\it Case 2):} $\gamma\in [\frac{\beta_1}{\beta_2}, 1).$}
In this case $\beta_1-\gamma\beta_2<0$ and, hence,  if 
$\zeta=(x,y): y\sim x^\gamma$, then 
$$
\P_{\zeta}(\zeta_1(1)=x+1,\, \zeta_2(1)=y)\sim \begin{cases}
x^{\beta_1-\gamma\beta_2}, &
\text{ if }\,\frac{\beta_1}{\beta_2}<\gamma<1,\\
 \frac{1}{2}, & \text{ if } \gamma=\frac{\beta_1}{\beta_2},
\end{cases}
$$
 as $\zeta\to\infty$.
Therefore, 
\begin{equation}
\label{nmb21}
\P_{\zeta}(S_{n}=0|S_{0}=1, ..., S_{n-1}=1)\sim 
\zeta_{1, n}^{\frac{\beta_1-\gamma\beta_2}{\gamma}N_n}=
e^{\frac{\eps(\beta_1-\gamma\beta_2)}{\gamma}\log(\zeta_{1,n})\zeta_{1,n}^{1-\gamma}},
\quad\text{if}\quad \frac{\beta_1}{\beta_2}<\gamma<1,
\end{equation}
and 
\begin{equation}
\label{nmb22}
\P_{\zeta}(S_{n}=0|S_{0}=1, ..., S_{n-1}=1)\sim 
\left(\frac{1}{2}\right)^{\zeta_{1,n}^{1-\gamma}},
\quad\text{if}\quad \gamma=\frac{\beta_1}{\beta_2}.
\end{equation}
Consequently, 
$
\lim_{\zeta\to\infty}
\P_{\zeta}\left(S_n=1\, \text{ for all }\,  n\geq 0\right)=1
$
 in  both $\frac{\beta_1}{\beta_2}<\gamma<1$
and $\gamma=\frac{\beta_1}{\beta_2}$ cases, 
as required.
The lemma is  proved.
\end{proof}

\subsection{Proofs of Theorems~\ref{main-sym} and~\ref{main-asym}}
\label{proof-reflection-theorems}

\begin{proof}[Proof of  Theorem~\ref{main-sym}]

Recall quantities $\alpha_k,\, k\geq 0$ defined in~\eqref{alpha-k}. It follows from~\eqref{alpha_k<alpha_k+1}, that 
for a given $\alpha\in (0, \alpha_{cr})$ there exists $k\geq 1$, such that
$\alpha\in[\alpha_{k-1}, \alpha_k)$. It is easy to verify  that 
$\alpha\in[\alpha_{k-1}, \alpha_k)$ is equivalent to~\eqref{k1-sym}.
Similarly, given $\gamma>1$, one can determine a unique $k$, such 
that $\gamma\in(\max(1, \gamma_{k}), \gamma_{k-1}]$, or, equivalently,  $k$ 
satisfying~\eqref{k2-sym}.
It follows from Theorem~\ref{T1a} (for the BB model without 
reflection) that, with probability one, the DTMC $\zeta(t)$  hits either the lower, or the upper boundary.
If follows then from  Lemmas~\ref{Cor-A1},~\ref{0+k} and~\ref{gamma-k}, that, with probability 
one, either the event $A_{\alpha, k_1}$, or the 
event $\widetilde{A}_{\gamma, k_2}$ occurs. 
It is left to note that these events are disjoint, which finishes the proof.
\end{proof}

\begin{proof}[Proof of Theorem~\ref{main-asym}]
The proof of Theorem~\ref{main-asym} is similar   to the proof of Theorem~\ref{main-sym}.
It only needs to be complemented by using  Lemma~\ref{Cor2-bisector}, 
Remark~\ref{Cor1-bisector} and Lemma~\ref{LT4ba}, where appropriate.
For example,  Lemma~\ref{LT4ba} is needed to deal with 
 the case of the upper boundary given by $y=x^{\gamma}$ for $\alpha_{cr}<\gamma\leq 1$.
We skip further details.
\end{proof}

\section{Open problem}
\label{open}

In this paper we have given a complete classification of the long term behavior 
of the BB model with reflection in the domain 
 $Q_{\varphi, \psi}$ (defined in~\eqref{domain}), 
where the lower and the upper boundaries are given 
by $\varphi(x)=x^{\alpha}$ and $\psi(x)=x^{\gamma}$ respectively for some 
$0<\alpha<\alpha_{cr}<\gamma$.
Specifically, we have shown that, for any initial state $\zeta\in Q_{\varphi, \psi}$, with probability one, 
 the DTMC $\zeta(t)$ eventually (i.e. after a finite time depending on the initial condition)
confines either to the strip $U_{\alpha, k_1}$ (along the lower boundary), or to the strip
 $\widetilde{U}_{\gamma, k_2}$ (along the upper boundary), where integers  $k_1$ and $k_2$ 
(interpreted as widths of the strips)
 can be calculated for a given set of the model parameters. We would like to stress that 
 widths of  both  strips are {\it finite}. 

We believe in a rather different behavior of the process in the case, when 
$\varphi(x)=\nu_{cr}x^{\alpha_{cr}}-b_1x^{\delta_1}$ and
 $\psi(x)=\nu_{cr}x^{\alpha_{cr}}+b_2x^{\delta_2}$
for some $b_1, b_2>0$ and $0<\delta_1, \delta_2<\alpha_{cr}$.
We conjecture that 
if both $0<\delta_1\leq \alpha_{cr}/2$ and $0<\delta_1\leq \alpha_{cr}/2$, then, with probability one, 
the process will be reflected infinitely many times by both the  lower and the upper boundary, i.e. 
 the process "visits"  {\it both} boundaries infinitely many times.
On the other hand, if  $\alpha_{cr}/2<\delta_1\leq \alpha_{cr}$ and 
$\alpha_{cr}/2<\delta_2\leq \alpha_{cr}$, then, with probability one, the process 
eventually sticks to one of the boundaries.  At the same time, deviations of the process 
from the  "final" boundary are not uniformly bounded, so that 
the process 
 eventually confines  to an indefinitely expanding  strip along one of the boundaries.
The open problem of interest is to determine the asymptotic width of the strip.

\end{document}